\documentclass{amsart}[12pt,draft]

\usepackage{amsmath,amssymb,amscd,amsthm,indentfirst}
\usepackage{amsfonts,eufrak}
\usepackage[mathscr]{eucal}
\usepackage{cases}
\usepackage[all]{xy}
\usepackage{enumerate}
\usepackage{hyperref}

\swapnumbers
\newtheorem{prop}[equation]{Proposition}
\newtheorem{thm}[equation]{Theorem}
\newtheorem{cor}[equation]{Corollary}
\newtheorem{lem}[equation]{Lemma}

\theoremstyle{definition}
\newtheorem{defn}[equation]{Definition}

\newtheorem{remark}[equation]{Remark}

\newtheorem{exa}[equation]{Example}

\def\Mor{\mathrm{Mor}}

\def\Aut{\mathrm{Aut}}

\def\A{\ensuremath{\mathcal{A}}}

\def\Bar{\ensuremath{\mathcal{B}}}
\def\C{\ensuremath{\mathcal{C}}}

\def\F{\ensuremath{\mathcal{F}}}
\def\H{\ensuremath{\mathcal{H}}}

\def\L{\ensuremath{\mathcal{L}}}

\def\O{\ensuremath{\mathcal{O}}}

\def\FF{\ensuremath{\mathbb{F}}}

\def\ZZ{\ensuremath{\mathbb{Z}}}

\def\dash{\operatorname{-}}

\def\Tot{\operatorname{Tot}}
\def\CCh{\operatorname{CCh}}
\def\Aut{\mathrm{Aut}}

\def\Hom{\mathrm{Hom}}
\def\CohMack{\mathrm{CohMack}}

\def\im{\operatorname{Im}}

\def\plocz{\ZZ_{(p)}}

\def\Ab{\mathbf{Ab}}

\newcommand{\pcomp}[1]{{#1}^\wedge_p}

\newcommand{\Syl}{\operatorname{Syl}\nolimits}

\newcounter{saveequation}

\date{\today}
\title{A spectral sequence for fusion systems}
\author{Antonio D\'{i}az Ramos}
\address{Departamento de {\'A}lgebra, Geometr{\'\i}a y Topolog{\'\i}a,
Universidad de M{\'a}\-la\-ga, Apdo correos 59, 29080 M{\'a}laga,
Spain.}

\thanks{Partially supported by FEDER-MCI grant MTM2010-18089, Junta de Andaluc{\'\i}a grant FQM-213 and P07-FQM-2863.}
\email{adiaz@agt.cie.uma.es}

\begin{document}

\begin{abstract}We build a spectral sequence converging to the cohomology of a fusion system with a strongly closed subgroup. This spectral sequence is related to the Lyndon-Hochschild-Serre spectral sequence and coincides with it for the case of an extension of groups. Nevertheless, the new spectral sequence applies to more general situations like finite simple groups with a strongly closed subgroup and exotic fusion systems with a strongly closed subgroup. We prove an analogue of a result of Stallings in the context of fusion preserving homomorphisms and deduce Tate's $p$-nilpotency criterion as a corollary.
\end{abstract}

\pagestyle{plain}

\numberwithin{equation}{section}
\renewcommand{\theequation}{\thesection.\arabic{equation}}
\renewcommand{\thethmain}{\Alph{thmain}}
\renewcommand{\theenumi}{(\arabic{enumi})}
\renewcommand{\labelenumi}{(\arabic{enumi})}

\maketitle

\section{Introduction.}

Let $K\unlhd G$ be a normal subgroup of the finite group $G$ and consider the extension
$$
K\rightarrow G\rightarrow G/K.
$$
The Lyndon-Hochschild-Serre spectral sequence of this short exact sequence is an important tool to analyze the cohomology of $G$ with coefficients in the $\ZZ G$-module $M$. It has second page $E_2^{n,m}=H^n(G/K;H^m(K;M))$ with $G/K$ acting on $H^m(K;M)$ and converges to $H^{n+m}(G;M)$.

Our aim in this work is to construct a related spectral sequence in the context of fusion systems. This concept was originally introduced by Puig and developed by Broto, Levi and Oliver in \cite{BLO2}, to where we refer the reader for notation. It consists of a category $\F$ with objects the subgroups of a finite $p$-group $S$ and morphisms bounded by axioms that mimic properties of conjugation morphisms. 

In the setup of fusion systems the concept of short exact sequence is an evasive one: Let $\F$ be a fusion system over the $p$-group $S$. For a strongly $\F$-closed subgroup $T$ of $S$ there is a quotient fusion system $\F/T$  \cite[5.10]{controlCraven}. Nevertheless, in general there is no normal fusion subsystem of $\F$ that would play the role of the kernel of the morphism of fusion systems $\F\rightarrow \F/T$ \cite[8.11 ff.]{normalAschbacher}. So the answer to \cite[Conjecture 11]{conjectures} is negative and one cannot expect to construct a Lyndon-Hochschild-Serre spectral sequence for fusion systems. Here we are able to construct a spectral sequence that converges to the cohomology of $\F$, $H^*(\F;M)$, where $M$ is a $\plocz$-module with trivial action of $S$. Recall that $H^*(\F;M)$ is defined \cite[Section 5]{BLO2} as the following subring of $\F$-stable elements in $H^*(S;M)$:
$$
H^*(S;M)^\F=\{z\in H^*(S;M) | res(z)=\varphi^*(z) \textit{ for each $\varphi\in \Hom_\F(P,S)$}\},
$$
where $res:H^*(S;M)\rightarrow H^*(P;M)$ is restriction in cohomology.

\begin{thm}\label{teoremaA}
Let $\F$ be a fusion system over the $p$-group $S$, $T$ a strongly $\F$-closed subgroup of $S$ and $M$ a $\plocz$-module with trivial $S$-action. Then there is a first quadrant cohomological spectral sequence with second page
$$
E_2^{n,m}=H^n(S/T;H^m(T;M))^\F
$$
and converging to $H^{n+m}(\F;M)$.
\end{thm}

The notation ${}^\F$ for the second page will be fully described in Section \ref{sectioncohomology}, and must be thought as taking $\F$-stable elements in a similar way as explained for $H^*(\F;M)$ above. Consider for each subgroup $P$ of $S$ the Lyndon-Hochschild-Serre spectral sequence of the extension
$$
P\cap T\rightarrow P\rightarrow P/P\cap T \cong PT/T
$$
converging to $H^*(P;M)$. A morphism $\varphi\in \Hom_\F(P,Q)$ induces a morphism $\varphi^*$ between the spectral sequences corresponding to $Q$ and $P$. Hence we have a contravariant functor from $\F$ to the category of spectral sequences. Recall that a morphism in this category from $E'$ to $E''$ 
is a sequence of homomorphisms of differential bigraded $\plocz$-modules, $f_k\colon E'_k\to E''_k$, $k\geq 0$, 
such that $H(f_k)\cong f_{k+1}$. The inverse limit spectral sequence or spectral sequence of $\F$-stable elements has $E^{n,m}_2$ entry equal to $H^n(S/T;H^m(T;M))^\F$, i.e.,  the elements $z$ from 
$$
H^n(S/T;H^m(T;M))
$$
such that $\varphi^*(z)=res(z)$, where $\varphi\in \Hom_\F(P,S)$ and $res=\iota^*$ is restriction in cohomology for the inclusion $P\stackrel{\iota}\leq S$.
Hence  $H^*(S/T;H^*(T;M))^\F$ is a differential graded subalgebra of the differential graded algebra $H^*(S/T;H^*(T;M))$ and its differential is just restriction of the differential of the latter. This should be useful in computations. The theorem states that the abutment of this spectral sequence is $H^*(\F;M)$.

For the case of a normal subgroup $K\unlhd G$ and $\F=\F_S(G)$ with $S\in \Syl_p(G)$ we have two spectral sequences converging to $H^*(G;M)$. Here,  $M$ is a $\plocz$-module with trivial $G$-action (and hence trivial $S$-action). On the one hand, we have the Lyndon-Hochschld-Serre spectral sequence associated to $K\rightarrow G\rightarrow G/K$. On the other hand, we have the spectral sequence associated to $\F$ and the strongly $\F$-closed subgroup $T=K\cap S\in \Syl_p(K)$. In Section 5 we prove that the two spectral sequences are isomorphic. Note that, in particular, this shows that the Lyndon-Hochschild-Serre spectral sequence of the extension $K\rightarrow G\rightarrow G/K$ depends only on the intersection of $K$ with a Sylow $p$-subgroup of $G$. 

As an application of the spectral sequence in Theorem \ref{teoremaA} we prove an analogue of a result of Stallings. Meanwhile the original theorem deals with a group homomorphism, here we replace that notion by that of a \emph{fusion preserving} homomorphism. This is a group homomorphism $S_1\rightarrow S_2$ between the Sylow $p$-subgroups of two fusion systems $\F_1$ and $\F_2$ such that morphisms of $\F_1$ are transformed into morphisms of $\F_2$ (see Section \ref{section_Stallings}).

\begin{thm}[{\cite[p.170]{stallings}}]
Let $\F_i$ be a fusion system over the $p$-group $S_i$ for $i=1,2$ and let $\phi\colon S_1\rightarrow S_2$ be a fusion preserving homomorphism. If the induced map in cohomology $H^i(\F_2;\FF_p)\rightarrow H^i(\F_1;\FF_p)$ is isomorphism for $i=1$ and monomorphism for $i=2$ then $S_1/\O^p_{\F_1}(S_1)\cong S_2/\O^p_{\F_2}(S_2)$.
\end{thm}

The \emph{hyperfocal subgrop} of $\F_i$, $O^p_{\F_i}(S_i)$, ($i=1,2$) is defined as follows
$$
O^p_{\F_i}(S_i)=\langle [P,O^p(\Aut_{\F_i}(P))] | P \leq S_i\rangle.
$$
It is the smallest subgroup of $S_i$ such that the quotient of $\F_i$ over that subgroup is a $p$-group \cite{extensions}. Hence, the conclusion of the theorem is that the largest $p$-group quotients of $\F_1$ and $\F_2$ are isomorphic. For instance, when $\F_1$ and $\F_2$ are already $p$-groups, i.e., $\F_i=\F_{S_i}(S_i)$, $i=1,2$, the conclusion is that $S_1$ and $S_2$ are isomorphic. This particular case is a variant of Stallings' result by Evens \cite[7.2.4]{evens}. We can also deduce fusion system versions of another result of Evens and Tate's $p$-nilpotency criterion:

\begin{cor}[{\cite[7.2.5]{evens}}]
Let $\F$ be a fusion system over the $p$-group $S$. If the map $H^2(\F/E^p_\F(S);\FF_p)\rightarrow H^2(\F;\FF_p)$ is a monomorphism then $S/O^p_\F(S)$  is elementary abelian.
\end{cor}

Here, the \emph{elementary focal subgroup} of $\F$ is defined as $E^p_\F(S)=\Phi(S)O^p_\F(S)$ \cite{tateyoshida}, where $\Phi(S)$ is the Frattini subgroup of $S$. The conclusion of this corollary is that the largest $p$-group quotient of $\F$ is elementary abelian.

\begin{cor}[{\cite[Corollary p.109]{Tate}}]
Let $\F$ be a fusion system over the $p$-group $S$. If the restriction map $H^1(\F;\FF_p)\rightarrow H^1(S;\FF_p)$ is an isomorphism then $\F=\F_S(S)$.
\end{cor}

This last result was already proven in \cite{tateyoshida} using transfer for fusion systems and in \cite{nilpotency} by topological methods. Here the proof mimics Tate's cohomological original proof that relies on the five terms exact sequence associated to the Lyndon-Hochschild-Serre spectral sequence but using instead the spectral sequence of Theorem \ref{teoremaA}.

There are situations where the Lyndon-Hochschild-Serre spectral sequence is not applicable while the spectral sequence from Theorem \ref{teoremaA} can be utilized. For instance, a classical drawback of the Lyndon-Hochschild-Serre spectral sequence is that it cannot be applied to finite simple groups. Nevertheless there are finite simple groups that do have a strongly closed $p$-subgroup. In \cite{FloresFoote}, Flores and Foote classified all finite groups with a strongly closed $p$-subgroup. In particular, they stated which finite simple groups have a stronly closed $p$-subgroup. Notice that even if $\F$ is induced from a non-simple finite group $\F=\F_S(G)$ not every strongly closed $\F$-subgroup $T$ of $S$ is of the form $T=K\cap S$ for some normal subgroup $K\unlhd G$ \cite[Example 6.4]{normalAschbacher}. This describes another circumstances where Lyndon-Hochschild-Serre does not apply but Theorem \ref{teoremaA} does. As final  example of this situation consider exotic fusion system with a strongly closed $p$-subgroup. A family of such exotic fusion systems is described in \cite{drv}, where the authors classified all the fusion systems over $p$-groups of $p$-rank $2$ ($p$ odd). 

This opens a new range of cohomology computations that can be carried out, some of which the author intends to perform in a subsequent paper. The main limitation here is that the spectral sequence from Theorem \ref{teoremaA} requires they knowledge of the Lyndon-Hochschild-Serre spectral sequence of the extension of $p$-groups $T\rightarrow S\rightarrow S/T$, and these computations do not abound. 

\if false {
For trivial coefficients $M$ there is an equivalent geometric construction of this spectral sequence which consists of considering the classifying spaces sitting in the fibration
$$
BK\rightarrow BG\rightarrow B(G/K)
$$
and then applying the Serre spectral sequence \cite{AdemMilgram}.

As a particular example, if $T=O^p_\F(S)$, the hyperfocal subgroup of $\F$, then in \cite{extensions} the authors proved that there is a unique subsystem $\H$ over $O^p_\F(S)$ of $p$-power index and in \cite{tateyoshida} we proved \textbf{(although this is probably folklore for experts?)} that $\H$ is weakly normal in $\F$. In this case the quotient is exactly $\F/O^p_\F(S)=\F_{S/T}(S/T)$, a $p$-group. Moreover, according to \cite[4.4]{extensions} the classifying space of $\H$ is the universal covering space of the classifying space of $\F$, so we get the following fibration as $\pi_1(\pcomp{|\L_\F|})\cong S/O^p_\F(S)$:
$$
S/O^p_\F(S)\rightarrow \pcomp{|\L_\H|}\rightarrow \pcomp{|\L_\F|}.
$$
Because the classifying spaces have the same cohomology (with trivial local coefficients) as $\H$ and $\F$, in this case our analogue of the Lyndon-Hochschild-Serre spectral sequence should coincide with the Leray-Serre spectral sequence of the above fibration \textbf{ (write the same thing for $p'$-power index).}
}
\fi

\begin{remark}
Theorem \ref{teoremaA} holds for the wider class of $\F$-stable $\plocz S$-modules, i.e., for $\plocz S$-modules $M$ such that for any morphism $\varphi\colon P\rightarrow S$ in $\F$ and any $p\in P$ we have $\varphi(p)\cdot m=p\cdot m$. Also, the Lyndon-Hochschild-Serre spectral sequence of $K\unlhd G$ and the spectral sequence from Theorem \ref{teoremaA} for $\F=\F_S(G)$ and $T=S\cap K$ coincide for $G$-stable $\plocz G$-modules, i.e, for $\plocz G$-modules $M$ such that $g^{-1}hg\cdot m=h\cdot m$ for any $h,g\in G$.
\end{remark}

\textbf{Organization of the paper:}
In Section \ref{sectioncohomology}, $\F$-stable elements and Mackey functors are defined and some related results introduced. In Section \ref{section_amackey}, we describe a particular cohomological Mackey functor that will play a central role in the construction of the spectral sequence. In Section \ref{section_construction}, the spectral sequence is built and Theorem \ref{teoremaA}  is proven as Theorem \ref{maintheorem}. In Section \ref{section_examples} we compare the spectral sequence from Theorem \ref{teoremaA} to the Lyndon-Hochschild-Serre spectral sequence and we give an example. In Section \ref{section_Stallings} we prove Stallings' result and its corollaries.
 
\textbf{Acknowledgments:} I would like to thank A. Viruel for several fruitful conversations when developing this paper. Also, I am grateful to P. Symonds for showing me how to prove that in the normal subgroup case the two spectral sequences coincide (Theorem \ref{theorem:sscomparison}).

\section{Cohomology and $\F$-stable elements.}\label{sectioncohomology}

Throughout this section $\F$ denotes a fusion system over the $p$-group $S$. We start introducing some notation: If $A:\F\rightarrow \C$ is a contravariant functor and $\C$ is any category then by $\varphi^*$ we denote the value $A(\varphi)$ for $\varphi$ a morphism in $\F$. For $\varphi=\iota_P^S$, the inclusion of $P$ into $S$, we write $res:={\iota_P^S}^*$. If $\C$ is a complete category then we denote by $A^\F$ the inverse limit over $\F$ of this functor:
$$
A^\F:= \varprojlim_\F A.
$$
For the complete category $\CCh(\Ab)$ of (unbounded) cochain complexes we have the following favourable description of inverse limits:

\begin{lem}\label{leminverslimitstableelements}
Let $A:\F\rightarrow \CCh(\Ab)$ be a contravariant functor. Then:
$$
A^\F=A(S)^\F:=\{z\in A(S) | res(z)=\varphi^*(z) \textit{ for each $\varphi\in \Hom_\F(P,S)$}\}\subseteq A(S).
$$
\end{lem}

We call the elements in $A(S)^\F$ the \emph{$\F$-stable elements} in $A(S)$. For such a functor we can consider the cohomology $H^*(A^\F)=H^*(A(S)^\F)$ of $A(S)^\F\in \CCh(\Ab)$. Notice that we also have functors $H^n(A):\F\rightarrow \Ab$ obtained by taking cohomology at degree $n$. Hence we may also consider the inverse limits $H^*(A)^\F=H^*(A(S))^\F$.  We are interested in functors $A$ for which taking $\F$-stable elements and cohomology commute. We prove in this section (Proposition \ref{lemmacochainsforF}) that being a cohomological Mackey functor (Definition \ref{defMackey}) with values in $\plocz$-modules is sufficient. 

\begin{defn}\label{defMackey}
Let $\F$ be a saturated fusion system over the $p$-group $S$ and let $\A$ be an abelian category. A \emph{cohomological Mackey functor} for $\F$ over $\A$ is a pair of functors $(A,B)\colon \F\rightarrow \A$ with $A:\F\rightarrow \A$ contravariant and $B:\F\rightarrow \A$ covariant such that:
\begin{enumerate}
\item $A(P)=B(P)$ and $A(\varphi)=B(\varphi^{-1})$ for each $P\leq S$ and $\varphi\in \Hom_\F(P,\varphi(P))$.
\item (Identity) $A(c_p), B(c_p)\colon A(P)\rightarrow A(P)$ are the identity morphisms for every $p\in P\leq S$, where $c_p:P\rightarrow P$, $x\mapsto pxp^{-1}$ is conjugation by $p$.

\item (Double coset formula) $A(\iota^P_Q)\circ B(\iota^P_R)=\sum_{x\in Q\backslash P/R} B(\iota^Q_{Q\cap {}^xR})\circ A(\iota^{{}^xR}_{Q\cap {}^xR})\circ A({c_{x^{-1}}}_{|{}^xR})$    for $Q,R\leq P\leq S$, where $Q\backslash P/R$ are the double cosets.

\item (Cohomological) $B(\iota^Q_P)\circ A(\iota^Q_P):A(Q)\rightarrow A(Q)$ is multiplication by $|Q:P|$ for every $P\leq Q\leq S$.
\end{enumerate}

\end{defn}

See \cite{Mackey} for the classical definition of Mackey functor and of cohomological Mackey functor for finite groups.
\begin{remark}
In Definition \ref{defMackey} we have omitted the familiar conditions
\begin{itemize}
\item (Transitivity) $B(\iota_Q^R)\circ B(\iota_P^Q)=B(\iota_P^R)$, $A(\iota_P^Q)\circ A(\iota_Q^R)=A(\iota_P^R)$ for $P\leq Q\leq R\leq S$ and
\item (Conjugation)$B(\iota_P^Q)\circ A(\varphi_{|P})=A(\varphi) \circ B(\iota_{{}^{\varphi(P)}}^{{}^{\varphi(Q)}})$, $B(\varphi_{|P})\circ A(\iota_P^Q)= A(\iota_{{}^{\varphi(P)}}^{{}^{\varphi(Q)}})\circ B(\varphi)$ for $P\leq Q\leq S$, $\varphi\in \Hom_\F(Q,\varphi(Q))$.
\end{itemize}
In fact, they are consequence of the functoriality of $A$ and $B$ and of condition ($1$).
\end{remark}
 We will use several times along the paper that cohomology of finite groups is a cohomological Mackey functor. For a proof of this fact see \cite{brown} for example.

\begin{remark}\label{remarktransferisenough}
If the maps $B(\iota_P^Q)$ for the inclusions $\iota_P^Q:P\rightarrow Q$ with $P,Q\leq S$ are given (these maps are called \emph{transfer})  we can define $B$ as follows: For any morphism $\varphi\in \Hom_\F(P,Q)$ define $B(\varphi):A(P)\rightarrow A(Q)$ by $B(\varphi)=B(\iota_{\varphi(P)}^Q)\circ A({\tilde\varphi}^{-1})$ with $\tilde\varphi:P\stackrel{\cong}\rightarrow \varphi(P)$. Such a $B$ becomes functorial if for any $P\leq Q\leq R$ we have $B(\iota_Q^R)\circ B(\iota_P^Q)=B(\iota_P^R)$ and for any $P\leq Q\stackrel{\varphi}\rightarrow \varphi(Q)$ we have $B(\iota_P^Q)\circ A(\varphi_{|P})=A(\varphi)\circ B(\iota_{\varphi(P)}^{\varphi(Q)})$.
\end{remark}


Before proving the main result of this section we need to introduce $(G,H)$-bisets: sets with commuting free right $G$-action and free left $H$-action. Every $(G,H)$-biset $\Omega$ can be decomposed into a disjoint union of transitive $(G,H)$-bisets of the form
$$
H\times_\varphi G=H\times G/\sim,
$$
with $K\leq G$, $\varphi:K\rightarrow H$ a monomorphism and
$$
(h,kg)\sim (h\varphi(k),g)
$$
for $h\in H$, $g\in G$ and $k\in K$. A saturated fusion system gives rise to a special type of biset:

\begin{prop}[{\cite[Proposition 5.5]{BLO2}}]\label{BLO2stablebiset}
For any saturated fusion system $\F$ over a $p$-group $S$, there is an $(S,S)$-biset $\Omega$ with the following properties:
\begin{enumerate}[(a)]
\item Each transitive component of $\Omega$
 is of the form $S\times_\varphi S$ for some $P\leq S$ and $\varphi\in \Hom_\F(P,S)$.\label{BLO2stablebiseta}


\item For each $P\leq S$ and each $\varphi\in \Hom_\F(P,S)$, the $(P,S)$-biset $\Omega_P$ obtained by restricting the right action from  $S$ to $P$ and the $(P,S)$-biset $\Omega_\varphi$ obtained by restricting the right action from $S$ to $P$ via $\varphi$ are isomorphic as $(P,S)$-bisets.\label{BLO2stablebisetb}
\item $|\Omega|/|S|=1 (mod p)$.\label{BLO2stablebisetc}
\end{enumerate}
\end{prop}

We call such an $(S,S)$-biset an \emph{$\F$-stable $(S,S)$-biset}. 
Now let $(A,B)\colon \F\rightarrow \A$ be a cohomological Mackey functor for $\F$ over the abelian category $\A$. For each transitive $(Q,R)$-biset $ R\times_\varphi Q$ with $\varphi\in \Hom_\F(P,R)$, $P\leq Q\leq S$, $R\leq S$, we have the composition 

\begin{equation}\label{transitiveomegaacting}
A(R)\stackrel{A(\varphi)}\rightarrow A(P)=B(P)\stackrel{B(\iota_P^Q)}\rightarrow B(Q)=A(Q).
\end{equation}
For each $(Q,R)$-biset $\Omega$ with
$$
\Omega=\coprod R\times_\varphi Q
$$
we can define a map $A(\Omega):A(R)\rightarrow A(Q)$ by
\begin{equation}\label{omegaacting}
A(\Omega):=\sum B(\iota)\circ A(\varphi).
\end{equation}

\begin{lem}\label{lemmabisets} Let $(A,B):\F\rightarrow \A$ be a cohomological Mackey functor. Then
\begin{enumerate}
\item For each transitive $(Q,R)$-biset  $ R\times_\varphi Q$ the morphism (\ref{transitiveomegaacting}) depends only on the isomorphism class of $ R\times_\varphi Q$ as  $(Q,R)$-biset.
\item For any $(Q,R)$-biset $\Omega$ the morphism (\ref{omegaacting}) depends only on the isomorphism class of $\Omega$ as  $(Q,R)$-biset.
\item For any $(Q,R)$-biset $\Omega$ and any monomorphism $\psi:P\rightarrow Q$ we have 
$$
A(\psi)\circ A(\Omega)=A(\Omega_\psi),
$$
where $\Omega_\psi$ is the $(P,R)$-biset obtained by restricting the right action of $\Omega$ from $Q$ to $P$ via $\psi$.
\item If $\A=\plocz\dash mod$ and $\Omega$ is an $\F$-stable $(S,S)$-biset then
$$
A(S)^\F=\im(A(\Omega):A(S)\rightarrow A(S)).
$$
\end{enumerate}
\end{lem}
\begin{proof}
Proof of ($1$): The transitive $(Q,R)$-bisets $R\times_{\varphi_1} Q$ and $R\times_{\varphi_2} Q$ with $\varphi_1:P_1\rightarrow R$, $\varphi_2:P_2\rightarrow R$, $P_1,P_2\leq Q$ are isomorphic as $(Q,R)$-bisets if and only if there exist elements $q\in Q$ and $r\in R$ such that the following diagram commutes:
$$
\xymatrix{
P_1\ar[d]^{c_q}\ar[r]^{\varphi_1}& R\ar[d]^{c_r}\\
P_2\ar[r]^{\varphi_2}& R.
}
$$
Hence both squares in the following diagram commute:
$$
\xymatrix{
A(R)\ar[r]^{A(\varphi_1)}&A(P_1)\ar[r]^{B(\iota_{P_1}^Q)}\ar@<1ex>[d]^{B(c_q)}&A(Q)\ar[d]^{B(c_q)}\\
A(R)\ar[u]^{A(c_r)}\ar[r]^{A(\varphi_2)}&A(P_2)\ar@<1ex>[u]^{A(c_q)}\ar[r]^{B(\iota_{P_2}^Q)}&A(Q).
}
$$
Using ($1$) and ($2$) from Definition \ref{defMackey} one finds out that 
$$
B(\iota_{P_1}^Q)\circ A(\varphi_1)=B(\iota_{P_2}^Q)\circ A(\varphi_2).
$$

Proof of ($2$): any automorphism of $\Omega$ permutes its transitive components via isomorphisms. So we may apply ($1$) from the Lemma to each component.

Proof of ($3$): write $\Omega$ as a disjoint union of transitive $(Q,R)$-bisets $\Omega=\coprod R\times_\varphi Q$. The transitive $(Q,R)$ biset $R\times_\varphi Q$ with $\varphi:K\rightarrow R$, $K\leq Q$ decomposes as a $(P,R)$-biset with $P$ acting via $\psi$ as follows:
$$
R\times_\varphi Q=\bigcup_{q\in \psi(P)\backslash Q/K} R\times_{\varphi\circ c_{q^{-1}}\circ \psi} P,
$$
with $P\geq P\cap\psi^{-1}({}^qK) \stackrel{\psi_{|}}\rightarrow
\psi(P)\cap {}^qK \stackrel{c_{q^{-1}}}\longrightarrow 
K\stackrel{\varphi}\rightarrow R$. Hence,
$$
A(\Omega_\psi)=\sum_\varphi \sum_{q\in \psi(P)\backslash Q/K} B(\iota^P_{P\cap\psi^{-1}({}^qK)})\circ A(\varphi\circ c_{q^{-1}}\circ \psi_{|}).
$$
Using functoriality of $A$ and $B$  we get
$$
A(\Omega_\psi)=A(\tilde{\psi})\circ(\sum_\varphi \sum_{q\in \psi(P)\backslash Q/K} B(\iota^{\psi(P)}_{\psi(P)\cap{}^qK})\circ A(c_{q^{-1}})\circ A(\varphi)),
$$
with $\tilde{\psi}:P\stackrel{\cong}\rightarrow \psi(P)$. Now  the Mackey decomposition ($3$) from Definition \ref{defMackey} gives
$$
A(\Omega_\psi)=\sum_\varphi A(\tilde{\psi}) \circ A(\iota^Q_{\psi(P)})\circ B(\iota^Q_K)\circ A(\varphi)=A(\psi)\circ A(\Omega).
$$

Proof of ($4$): Let $z\in A(S)$. We want to see that $A(\Omega)(z)\in A(S)^\F$. So let $\psi$ be a morphism in $\Hom_\F(P,S)$. Then
$$
A(\psi)(A(\Omega)(z))=(A(\psi)\circ A(\Omega))(z)=A(\Omega_\psi)(z)
$$
by ($3$). By \ref{BLO2stablebiset}(\ref{BLO2stablebisetb}), the $(P,S)$-bisets $\Omega_\psi$ and $\Omega{}_{\iota_P^S}=\Omega_P$ are isomorphic as $(P,S)$-bisets. Then by ($2$) we have $A(\Omega{}_\psi)=A(\Omega{}_{\iota_P^S})$. Hence,
$$
A(\psi)(A(\Omega)(z))=A(\Omega{}_{\psi})(z)=A(\Omega{}_{\iota_P^S})(z)=A({\iota_P^S})(A(\Omega)(z))
$$
by ($3$). Thus $A(\Omega)(z)\in A(S)^\F$.

Now let $z\in \A(S)^\F$. Then
$$
A(\Omega)(z)=\sum B(\iota)(A(\varphi)(z))=\sum B(\iota)(A(\iota)(z))
$$
as $z$ is $\F$-stable. Now by ($4$) of Definition \ref{defMackey} we get 
$$
A(\Omega)(z)=(\sum |S:P|)\cdot z
$$
and by \ref{BLO2stablebiset}(\ref{BLO2stablebisetc}) the number $q=(\sum |S:P|)=|\Omega|/|S|$ is a $p'$-number. So $A(\Omega)(\frac{z}{q})=z$ and hence $z\in \im A(\Omega)$.
\end{proof}

For a fusion system $\F$ over the $p$-group $S$ denote by 
$\CohMack_{\plocz}(\F)$ the abelian category with objects the cohomological Mackey functors with values in $\plocz\dash mod$ and morphisms the natural transformations commuting with both the contravariant and covariant parts. This means that if $(A,B)$ and $(A',B')$ are cohomological Mackey functors, a morphism $\eta$ between them consists of a morphism of $\plocz$-modules $\eta_P\colon A(P)\rightarrow A'(P)$ for each $P\leq S$ such that for $\varphi\in \Hom_\F(P,Q)$ we have $A'(\varphi)\circ \eta_Q=\eta_P\circ A(\varphi)$ and $\eta_Q\circ B(\varphi)=B'(\varphi)\circ \eta_P$.

\begin{lem} \label{lemmainvariantsexactonMackey}
Let $\F$ be a fusion system over the $p$-group $S$. Then the functor $\CohMack_{\plocz}(\F)\stackrel{(-)^\F}\rightarrow \plocz\dash mod$ sending $(A,B)\mapsto A^\F$ is exact.

\end{lem}
\begin{proof}
Let 
$$
0\Rightarrow (A_1,B_1)\Rightarrow (A_2,B_2)\stackrel{\eta}\Rightarrow (A_3,B_3)\Rightarrow 0
$$
be an exact sequence in $\CohMack_{\plocz} (\F)$. We want to prove that $0\rightarrow A_1^\F\rightarrow A_2^\F\stackrel{\eta^\F}\rightarrow A_3^\F\rightarrow 0$ is exact in $\plocz\dash mod$. The non-trivial equality to prove is that the arrow $A_2^\F\rightarrow A_3^\F$ is an epimorphism. So let $z$ be an $\F$-stable element in $A_3(S)$. Fix an $(S,S)$-biset $\Omega$ satisfying the properties of Proposition \ref{BLO2stablebiset}. By Lemma \ref{lemmabisets}(4) there exists an element $z'\in A_3(S)$ with $z=A_3(\Omega)(z')$. By hypothesis, the map
$$ 
A_2(S)\stackrel{\eta_S}\rightarrow A_3(S)
$$
is an epimorphism and hence there exists an element $y'\in A_2(S)$ with $\eta_{S}(y')=z'$. By Lemma \ref{lemmabisets}(4) again we have that $y\stackrel{def}= A_2(\Omega)(y')$ belongs to $A_2^\F$. Because $\eta$ commutes with the covariant and contravariant parts of $(A_2,B_2)$ and $(A_3,B_3)$ is easy to see that
$$
\eta^\F(y)=\eta^\F(A_2(\Omega)(y'))=A_3(\Omega)(\eta^\F(y'))=A_3(\Omega)(z')=z.
$$
\end{proof}

\begin{prop} \label{lemmacochainsforF}
Let $\F$ be a fusion system over $S$ and let $(A,B):\F\rightarrow \CCh(\plocz)$ be a cohomological Mackey functor. Then
$$
H^*(A(S)^\F)\cong H^*(A(S))^\F.
$$
\end{prop}
\begin{proof}
This is a consequence of the previous Lemma and of the well known fact that cohomology commutes with exact functors.
\end{proof} 

\begin{remark}\label{remarkMackeycochainsdonotknow}
Let $\F$ be a fusion system over the $p$-group $S$ and let $M$ be a trivial $\plocz S$-module. By \cite[Section 5]{BLO2} the cohomology of $\F$ is defined as $H^*(\F;M)=H^*(S;M)^\F$ where $H^*(\cdot;M):\F\rightarrow \plocz$-modules is the cohomological Mackey functor with values $H^*(P;M)$. If one could choose cochains $C^*(\cdot;M):\F\rightarrow \CCh(\plocz)$ such that $C^*(\cdot;M)$ was the contravariant part of a cohomological Mackey functor then Proposition \ref{lemmacochainsforF} would give the computational-purposes formula
$$
H^*(\F;M)=H^*(C^*(S;M)^\F).
$$

In the next section some problems related to the functoriality of cochains will become apparent.
\end{remark}

\section{A Mackey functor.}\label{section_amackey}
Let $\F$ be a fusion system over the $p$-group $S$, $T$ a strongly $\F$-closed subgroup of $S$ and $M$ a $\plocz$-module with trivial $S$-action. In this section we prove that for every $n,m\geq 0$ the functor $H^{n,m}:\F\rightarrow \plocz\dash mod$ sending the subgroup  $P\leq S$ to $H^n(P/P\cap T;H^m(P\cap T;M))$ is the contravariant part of a cohomological Mackey functor $\F\rightarrow \CCh^2(\plocz)$ with values in double (cochain) complexes (Definition \ref{defMackey}). Here, by double complexes we mean the abelian category with objects families  of $\plocz$-modules $\{A^{n,m}\}_{n,m\in \ZZ}$ together with maps $d^h$ (horizontal differential) and $d^v$ (vertical differential)
$$
d^h\colon A^{n,m}\rightarrow A^{n+1,m}\textit{ and }d^v\colon A^{n,m}\rightarrow A^{n,m+1},
$$
such that $d^h d^h=d^vd^v=d^hd^v+d^vd^h=0$. A morphism from  $\{A^{n,m}\}_{n,m\in \ZZ}$ to $\{{A'}^{n,m}\}_{n,m\in \ZZ}$ is a family of maps of $\plocz$-modules $\{A^{n,m}\rightarrow {A'}^{n,m}\}_{n,m\in \ZZ}$ that commute with horizontal and vertical differentials.

For $P\leq S$ denote by $\overline P$ the group $P/P\cap T$. The bar resolutions $\Bar^*_P$ and $\Bar^*_{\overline P}$ for $P$ and $\overline P$ respectively are projective resolutions of the trivial module $\plocz$ over $\plocz P$ and $\plocz {\overline P}$ respectively. Recall that the bar resolution is functorial (covariant) over finite groups and homomorphisms. Define $A^{*,*}(P)$ as the double complex associated to the short exact sequence
$$
0\rightarrow P\cap T\rightarrow P\rightarrow \overline P\cong PT/T\rightarrow 0.
$$
More precisely, for $n\geq 0$ and $m\geq 0$, we define
$$
A^{n,m}(P)=\Hom_P(\Bar_{\overline P}^n\otimes \Bar_P^m,M)
$$
where $P$ acts on  $\Bar_{\overline P}^n\otimes \Bar_P^m$ by $p(y\otimes x)=\overline p y \otimes px$ for $y\in \Bar^n_{\overline P}$ and $x\in \Bar^m_P$.

As the action of $P$ on $M$ is trivial the cochains in $A^{n,m}(P)$ are the homomorphisms $f\in \Hom(\Bar_{\overline P}^n\otimes \Bar_P^m,M)$ such that
$$
f(\overline p y \otimes px)=f(y\otimes x)
$$
for all $y\in \Bar^n_{\overline P}$, $x\in \Bar^m_P$ and $p\in P$.

To obtain a double complex we consider the following horizontal and vertical differentials for $f\in A^{n,m}(P)$
$$
d^h(f)(y\otimes x)=(-1)^{n+m+1}f(d(y)\otimes x)\textit{, $y\in \Bar^{n+1}_{\overline P}$, $x\in \Bar^m_P$ }
$$
and
$$
d^v(f)(y\otimes x)=(-1)^{m+1}f(y\otimes d(x))\textit{, $y\in \Bar^n_{\overline P}$, $x\in \Bar^{m+1}_P$, }
$$
where we are using the differential $d$ of the complexes $\Bar^*_{\overline P}$ and $\Bar^*_P$. We choose the signs 
as in \cite[XI.10.1]{homology} to ensure that $d^hd^v+d^vd^h=0$. We will obtain the functor $H^{n,m}$ by taking vertical cohomology followed by horizontal cohomology in $A^{n,m}$.

To define $A$ on morphisms notice that any morphism $\varphi\in \Hom_\F(P,Q)$ takes $P\cap T$ to $Q\cap T$ as $T$ is strongly $\F$-closed. Hence it induces a homomorphism 
$$
\overline{\varphi}:\overline P\rightarrow \overline Q.
$$ 
Thus for any $\varphi\in \Hom_\F(P,Q)$ we may define 
$$
A^{n,m}(Q)\stackrel{A^{n,m}(\varphi)}\rightarrow A^{n,m}(P)
$$ 
mapping  the cochain $f\in A^{n,m}(Q) $ to the cochain in $A^{n,m}(P)$ that takes $y\in \Bar_{\overline P}^n$ and $x\in \Bar_P^m$ to 
$$
f(\Bar^n(\overline{\varphi})(y)\otimes \Bar^m(\varphi)(x)),
$$
where $\Bar^n(\overline{\varphi})$ and $\Bar^m(\varphi)$ are the usual morphisms between bar resolutions. They commute with differentials and satisfy
$$
\Bar^n(\overline\varphi)(\overline p\cdot y)=\overline\varphi(\overline p)\cdot \Bar^n(\overline\varphi)(y)
$$
for every $y\in \Bar^n_{\overline P}$ and every $\overline p\in \overline P$ and 
$$
\Bar^m(\varphi)(p\cdot x)=\varphi(p)\cdot \Bar^m(\varphi)(x)
$$
for every $x\in \Bar^m_P$ and $p\in P$. It is straightforward that $A^{n,m}(\varphi)(f)\in A^{n,m}(P)$ and that the family of morphisms  $\{A^{n,m}(\varphi)\}_{n,m\geq 0}$ commutes with the horizontal and vertical differentials of the double complexes $A^{*,*}(Q)$ and $A^{*,*}(P)$.

\begin{remark}\label{remark_Cravenquotient}
By definition the fusion system $\F/T$ is defined over the $p$-group $S/T$. For $T\leq P,Q\leq S$ the morphisms in  $\Hom_{\F/T}(P/T,Q/T)$ are those homomorphisms $\overline{\psi}:P/T\rightarrow Q/T$ induced on the quotient from $\psi\in \Hom_\F(P,Q)$.

For $P,Q\leq S$ and $\varphi\in \Hom_\F(P,Q)$ we have a morphism $\overline{\varphi}\colon\overline{P}\rightarrow \overline{Q}$. Then we have a commutative diagram
$$
\xymatrix{
\overline{P}\ar[d]_{\cong}\ar[r]^{\overline{\varphi}}&\ar[d]^ {\cong}\overline{Q}\\
PT/T\ar[r]^{\overline{\varphi}}&QT/T
}
$$
where $\overline{\varphi}$'s are induced by $\varphi$ and where the vertical arrows are the natural isomorphisms. According to \cite[5.10]{controlCraven} bottom morphism $\overline{\varphi}$ belongs to $\F/T$, i.e., there exists $\psi\in \Hom_\F(PT,QT)$ such that the induced map $\overline{\psi}:PT/T\rightarrow QT/T$ coincides with the given one.
\end{remark}

\begin{remark}\label{remarkhpabisfunctor}
The construction of $A^{n,m}$ is clearly functorial and hence so far we have a contravariant functor $A^{*,*}:\F\rightarrow \CCh^2(\plocz)$ with values in double complexes. 
\end{remark}

Now we define $B^{n,m}(P)=A^{n,m}(P)$ for every $P\leq S$ and  $n,m\geq 0$. For each morphism $\varphi\in \Hom_\F(P,Q)$ we will define a morphism of double complexes $B^{n,m}(\varphi):A^{n,m}(P)\rightarrow A^{n,m}(Q)$. This will not make $B$ into a covariant functor $\F\rightarrow \CCh^2(\plocz)$ as the definition depends on a choice of representatives. Nevertheless, $B$ will become functorial once we pass to cohomology.

To define $B^{*,*}(\psi)$ on $\psi\in \Hom_\F(P,Q)$, write $\psi=\iota\circ \tilde{\psi}$, where $\tilde{\psi}:P\rightarrow \psi(P)$ is an isomorphism and $\iota$ is the inclusion $\psi(P)\leq Q$, and set \begin{equation}\label{eqnBongenericmorphism}
B^{*,*}(\psi)=B^{*,*}(\iota)\circ A^{*,*}({\tilde{\psi}}^{-1}).
\end{equation}
So we just need to define $B$ on inclusions.

So let $\iota$ be the inclusion between subgroups $P\leq Q$ of $S$. There are maps of $\plocz P$-chain complexes and of $\plocz \overline P$-chain complexes respectively
$$
\tau_*^{Q,P}:\Bar^*_Q\rightarrow \Bar^*_P\textit{   ,  } \overline \tau^{Q,P}_*:\Bar^*_{\overline Q}\rightarrow \Bar^*_{\overline P}
$$
built as in \cite[(D), page 82]{brown}. More precisely, the map $\tau^{Q,P}_*$ is induced by a map of left $P$-sets $Q\stackrel{\rho}\rightarrow P$ defined as follows: fix a set of representatives for the right cosets $P\backslash Q$, then $\rho(q)=q\overline{q}^{-1}$ where  $\overline{q}$ is the representative with $Pq=P\overline{q}$. The map $\overline \tau^{Q,P}_*$ is defined analogously choosing representatives for the right cosets $\overline P\backslash \overline Q$. These choices of representatives prevent $B^{p,q}$ from being functorial. 
%
%

We define the map 
$$
B^{n,m}(\iota)\colon \Hom_P(\Bar_{\overline P}^n\otimes \Bar_P^m,M)\rightarrow \Hom_Q(\Bar_{\overline Q}^n\otimes \Bar_Q^m,M)
$$
by
\begin{equation}\label{defdiagonaltwistedtransfer}
B^{n,m}(\iota)(f)(y\otimes x)=\sum_{w\in Q/P} f({\overline {\tau}^{Q,P}_n}{(\overline w}^{-1}y)\otimes \tau^{Q,P}_m(w^{-1}x)),
\end{equation}
where $w$ runs over a set of representatives of the left cosets $Q/P$. This formula can  be thought as a relative transfer formula for twisted coefficients. Clearly its definition does not depend on the representatives $w$ chosen and $B^{n,m}(\iota)(f)\in A^{n,m}(Q)$. Moreover, $B^{n,m}(\iota)$ commutes with both the horizontal and vertical differentials as $\tau_*$ and $\overline\tau_*$ do and so it is a map of double complexes.

\begin{remark}
By \cite{Sejongindex} there are finite groups $G$ and $\overline{G}$ such that $S$ is a $p$-subgroup of $G$ (not necessarily  a Sylow $p$-subgroup), $\overline{S}=S/T$ is a $p$-subgroup of $\overline G$ (not necessarily  a Sylow $p$-subgroup) and with $\F=\F_S(G)$ and $\F/T=\F_{\overline S}(\overline G)$. Let $B^*_G$ and $B^*_{\overline G}$ be the bar resolutions of $G$ and $\overline G$ respectively. Then we could have defined for $P\leq S$
$$
A^{n,m}(P)=\Hom_P(\Bar^n_{\overline G}\otimes \Bar^m_G,M),
$$
where $P$ acts on $\Bar^n_{\overline G}\otimes \Bar^m_G$ by restricting the actions of $G$ on $\Bar^*_G$ and of $\overline G$ on $\Bar^*_{\overline G}$. This means that $p(y\otimes x )=\overline{p}y\otimes px$ for $p\in P$. In this setup clearly one can define a functorial  $B^{n,m}$ on inclusions. On the other hand, to realize a morphism $\varphi:P\rightarrow Q$ we need to choose $g\in N_G(P,Q)$ with $\varphi=c_g$ and $\overline g\in N_{\overline G}(\overline P,\overline Q)$ with $\overline\varphi=c_{\overline g}$ and then define
$$
A^{n,m}(\varphi)(y\otimes x)=({\overline g}y\otimes gx).
$$
It is clear that in general $A^{n,m}$ defined this way will not be functorial on morphisms. If one could choose $A^{n,m}$ and $B^{n,m}$ such that $(A^{n,m},B^{n,m})\colon \F\rightarrow \CCh^2(\plocz)$ was a Mackey functor then the proof of Theorem \ref{maintheorem} would be simpler. 
\end{remark}

On each double complex $A^{*,*}(P)$ with $P\leq S$ we may take vertical cohomology followed by horizontal cohomology to obtain $H^*(\overline P;H^*(P\cap T;M))$ \cite[Equation (10.2), page 352]{homology}. For any homomorphism $\varphi\in \Hom_\F(P,Q)$ the maps $A^{*,*}(\varphi)$ and $B^{*,*}(\varphi)$ are maps of double complexes and hence they induce maps
$$
H^{n,m}(A)(\varphi): H^n(\overline Q;H^m(Q\cap T;M))\rightarrow H^n(\overline P;H^m(P\cap T;M))
$$
and
$$
H^{n,m}(B)(\varphi): H^n(\overline P;H^m(P\cap T;M))\rightarrow H^n(\overline Q;H^m(Q\cap T;M)).
$$

\begin{lem}\label{lemmahpqadecomposition}
For $\varphi:P\rightarrow Q$ the map $H^{n,m}(A)(\varphi)$ factors as
$$
\xymatrix{
H^n(\overline Q;H^m(Q\cap T;M))\ar[r]^{H^n(\overline\varphi)}&
H^n(\overline P;H^m(Q\cap T;M))\ar[r]^{H^m(\varphi)}&
H^n(\overline P;H^m(P\cap T;M))
}
$$
where 
\begin{itemize}
\item $H^n(\overline\varphi)$ is the map induced by $\overline\varphi$ in cohomology  with $H^m(Q\cap T;M)$-coefficients, 
\item $H^m(\varphi)$ is the map induced by the change of coefficients 
$$
H^m(\varphi):H^m(Q\cap T;M)\rightarrow H^m(P\cap T;M).
$$
This map is a map of $\plocz \overline P$-modules where
$\overline P$ acts on $H^m(Q\cap T;M)$ via $\overline P\stackrel{\overline\varphi}\rightarrow \overline{\varphi(P)}\leq \overline Q$.
\end{itemize}
\end{lem}
\begin{proof}
By construction.
\end{proof}
\begin{lem}\label{lemmahpqbdecomposition}
If $P\leq Q$ and $\iota$ denotes the inclusion then the map $H^{n,m}(B)(\iota)$ factors as
$$
\xymatrix{
H^n(\overline P;H^m(P\cap T;M))\ar[r]^{H^m(tr')}&
H^n(\overline P;H^m(Q\cap T;M))\ar[r]^{H^n(tr)}&
H^n(\overline Q;H^m(Q\cap T;M))
}
$$
where 
\begin{itemize}
\item $H^n(tr)$ is the transfer map in cohomology with $H^m(Q\cap T;M)$-coefficients, 
\item $H^m(tr')$ is the map induced by the change of coefficients given by the transfer map in cohomology 
$$
H^m(tr'):H^m(P\cap T;M)\rightarrow H^m(Q\cap T;M).
$$
This map is a map of $\plocz \overline P$-modules where $\overline P$ acts on $H^m(Q\cap T;M)$ via $\overline P\leq \overline Q$.
\end{itemize}
\end{lem}
\begin{proof}
Choose representatives $z_i\in\overline Q$ of the left cosets $\overline Q/\overline P$ and representatives $t_j\in Q\cap T$ of the left cosets $(Q\cap T)/(P\cap T)$. Choose also representatives $q_k\in Q$ of the left cosets $\overline{Q}=Q/(Q\cap T)$. Then each $z_i\in \overline{Q}$ is represented as $z_i=\overline{q_{k_i}}$ for a unique $k_i$. It is an exercise to prove that the set of elements of $Q$ $q_{k_i}t_j$ for all $i$ and $j$ is a set of representatives of $Q/P$.
Then we can rewrite Equation (\ref{defdiagonaltwistedtransfer}) as
$$
\sum_{z_i\in \overline Q/\overline P}\sum_{t_j\in (Q\cap T)/(P\cap T)} f({\overline {\tau}^{Q,P}_n}{(\overline {q_{k_i}t_j}}^{-1}y)\otimes \tau^{Q,P}_m((q_{k_i}t_j)^{-1}x)),
$$
Because $t_j\in Q\cap T$ then $\overline {q_{k_i}t_j}=\overline {q_{k_i}}$ and the formula simplifies to
$$
\sum_{z_i\in \overline Q/\overline P}\sum_{t_j\in (Q\cap T)/(P\cap T)} f({\overline {\tau}^{Q,P}_n}{(\overline {q_{k_i}}}^{-1}y)\otimes \tau^{Q,P}_m(t_j^{-1}{q_{k_i}}^{-1}x)).
$$
This coincides with the composition in the statement of the lemma.\end{proof}

Lemma \ref{lemmahpqbdecomposition} proves in particular that the definition of $H^{n,m}(B)(\iota)$ does not depend on the representatives chosen to construct the maps $\tau^{Q,P}_*$ and $\overline{\tau}^{Q,P}_*$. (Although $B^{n,m}(\iota)$ do depends on them.)

\begin{cor}\label{corhpqbisfunctor}
For $n,m\geq 0$ the assignment 
$$
H^{n,m}(B):\F\rightarrow \plocz\dash mod
$$
taking $P$ to $H^n(\overline P;H^m(P\cap T;M))$ and taking $\varphi\in \Hom_\F(P,Q)$ to $H^{n,m}(B)(\varphi)$ is a functor.
\end{cor}
\begin{proof}
By Remarks \ref{remarktransferisenough} and \ref{remarkhpabisfunctor} and Equation (\ref{eqnBongenericmorphism}) it is enough to prove that for  any $P\leq Q\leq R$ we have 
$$
H^{n,m}(B)(\iota_Q^R)\circ H^{n,m}(B)(\iota_P^Q)=H^{n,m}(B)(\iota_P^R)
$$ and for any $P\leq Q\stackrel{\varphi}\rightarrow \varphi(Q)$ we have 
$$
H^{n,m}(B)(\iota_P^Q)\circ H^{n,m}(A)(\varphi_{|P})=H^{n,m}(A)(\varphi)\circ H^{n,m}(B)(\iota_{\varphi(P)}^{\varphi(Q)}).
$$
We can check both conditions at the level of cochains: For the first condition, the definitions (\ref{defdiagonaltwistedtransfer}) of $B^{n,m}(\iota_P^Q)$, $B^{n,m}(\iota_Q^R)$ and $B^{n,m}(\iota_P^R)$ depend upon choices of representatives for the right cosets $P\backslash Q$ and $\overline P\backslash \overline Q$, $Q\backslash R$ and $\overline Q\backslash \overline R$ and $P\backslash R$ and $\overline P\backslash \overline R$ respectively. Fix choices of representatives for the first four right cosets. Then the bijections $P\backslash Q\times Q\backslash R\rightarrow P\backslash R$ and $\overline P\backslash \overline Q\times \overline Q\backslash \overline R\rightarrow \overline P\backslash \overline R$ provide choices for the last two right cosets. With these choices we have 
$$
B^{n,m}(\iota_Q^R)\circ B^{n,m}(\iota_P^Q)=B^{n,m}(\iota_P^R).
$$

For the second condition, the maps $B^{n,m}(\iota_P^Q)$ and $B^{n,m}(\iota_{\varphi(P)}^{\varphi(Q)})$ depend on choices of representatives for the right cosets $P\backslash Q$ and $\overline P\backslash \overline Q$, and $\varphi(P)\backslash \varphi(Q)$ and $\overline {\varphi(P)}\backslash \overline {\varphi(Q)}$ respectively. Fix representatives in $P\backslash Q$ and $\overline P\backslash \overline Q$ and force the other choices via the bijections $P\backslash Q \stackrel{\varphi}\rightarrow \varphi(P)\backslash \varphi(Q)$ and $\overline P\backslash \overline Q\stackrel{\overline\varphi}\rightarrow\overline {\varphi(P)}\backslash \overline {\varphi(Q)}$. Then we have
$$
B^{n,m}(\iota_P^Q)\circ A^{n,m}(\varphi_{|P})=A^{n,m}(A)(\varphi)\circ B^{n,m}(\iota_{\varphi(P)}^{\varphi(Q)}).
$$
\end{proof}

\begin{prop}\label{hpqisMackey}
For each $p,q\geq 0$ the functor $\F\rightarrow \plocz\dash mod$ with values 
$$
H^p(\overline P;H^q(P\cap T;M)
$$ and taking $\varphi\in \Hom_\F(P,Q)$ to $H^{p,q}(A)(\varphi)$ is a cohomological Mackey functor with covariant part taking $\varphi\in \Hom_\F(P,Q)$ to $H^{p,q}(B)(\varphi)$.
\end{prop}
\begin{proof}
Property (1) from Definition \ref{defMackey} holds by Equation (\ref{eqnBongenericmorphism}). Property (2) follows from property (1), the well known fact that conjugation induces the identity on cohomology, from Lemma \ref{lemmahpqadecomposition} and from $\overline{c_p}=c_{\overline p}$ for $p\in P\leq S$. Now we check property ($3$), also known as Mackey condition or double coset formula. So let $Q,R\leq P\leq S$. We will prove this condition at the level of cochains, i.e.:
$$
A^{n,m}(\iota^P_Q)\circ B^{n,m}(\iota^P_R)=\sum_{x\in Q\backslash P/R} B^{n,m}(\iota^Q_{Q\cap {}^xR})\circ A^{n,m}(\iota^{{}^xR}_{Q\cap {}^xR})\circ A^{n,m}({c_{x^{-1}}}_{|{}^xR}).
$$
So let $f\in A^{n,m}(R)=\Hom_R(\Bar_{\overline R}^n\otimes \Bar_R^m,M)$, $y\in \Bar^n_{\overline{Q}}$ and $x\in \Bar^m_{Q}$. Then:
$$
A^{n,m}(\iota^P_Q)(B^{n,m}(\iota^P_R)(f))(y\otimes x)=B^{n,m}(\iota^P_R)(f)(\overline{\iota^P_Q}(y)\otimes \iota^P_Q(x))=B^{n,m}(\iota^P_R)(f)(y\otimes x).
$$
This equals
$$
\sum_{w\in P/R} f(\overline{\tau}^{P,R}_n({\overline w}^{-1}y)\otimes \tau^{P,R}_m(w^{-1}x)),
$$
where $w$ runs over a set of representatives of the left cosets $P/R$, $\overline{\tau}^{P,R}_n:\Bar^n_{\overline P}\rightarrow \Bar^n_{\overline R}$ and $\tau^{P,R}_m:\Bar^m_P\rightarrow \Bar^m_R$. Now we let $Q$ acts on the left on $P/R$ and we group together the terms corresponding to a given $Q$-orbit in $P/R$:
$$
\sum_{p\in Q\backslash P/R} \sum_{q\in Q/Q\cap {}^p R}  f(\overline{\tau}^{P,R}_n({\overline {qp}}^{-1}y)\otimes \tau^{P,R}_m((qp)^{-1}x)),
$$
where now $p$ runs over a set of representatives for the double cosets $Q\backslash P/R$ and $q$ runs over a set of representatives of the left cosets $Q/Q\cap {}^n R$. This equals
$$
\sum_{p\in Q\backslash P/R} \sum_{q\in Q/Q\cap {}^p R}  f(\overline{\tau}^{P,R}_n({\overline p}^{-1}{\overline q}^{-1}y)\otimes \tau^{P,R}_m(p^{-1}q^{-1}x)).
$$
The right-hand side of the Mackey formula is
$$
\sum_{p\in Q\backslash P/R} \sum_{q\in Q/Q\cap {}^p R}  f({\overline p}^{-1}\overline{\tau}^{Q,Q\cap {}^p R}_n({\overline q}^{-1}y){\overline p}\otimes p^{-1}\tau^{Q,Q\cap {}^p R}_q(q^{-1}x)p)
$$
with $\overline{\tau}^{Q,Q\cap {}^p R}_n:\Bar^n_{\overline Q}\rightarrow \Bar^n_{\overline{ Q\cap {}^p R}}$, $\tau^{Q,Q\cap {}^p R}_m:\Bar^m_Q\rightarrow \Bar^m_{Q\cap {}^p R}$ and where we have realized ${c_{p^{-1}}}_{|{}^pR}$ at the level of cochains as
$$
A^{n,m}({c_{p^{-1}}}_{|{}^pR})(y\otimes x)=({\overline p}^{-1}y{\overline p}\otimes p^{-1}xp).
$$

The map $\tau^{P,R}_m$ depends on a choice of representatives for the right cosets $R\backslash P$. Analogously, for any representative $p\in Q\backslash P/R$, the map $\tau^{Q,Q\cap {}^p R}_m$ is built out of a set of representatives of $Q\cap {}^p R\backslash Q$. We want to choose representatives of $R\backslash P$ and of $Q\cap{}^p R\backslash Q$ for each double coset $QpR$ such that
$$
\xymatrix{
Q\ar[r]^<<<<<<{\rho_p}\ar[d]_{q\mapsto p^{-1}q}& Q\cap{}^pR\ar[d]^{q\mapsto p^{-1}q p}\\
P\ar[r]^{\rho}& R
}
$$
commutes for each double coset $QpR$. For this is enough to choose arbitrary representatives $q$ of $Q\cap{}^p R\backslash Q$ for each double coset $QpR$ and build the representatives in $R\backslash P$ via the bijection
$$
\bigsqcup_{p\in Q\backslash P/R} Q\cap {}^p R\backslash Q\rightarrow R\backslash P
$$
that takes $(Q\cap {}^p R)q$ to $Rp^{-1}q$. The same argument for  $\overline\tau^{P,R}_n$  and the maps $\overline{\tau}^{Q,Q\cap {}^p R}_n$ finishes the proof of property ($3$).

To prove property ($4$) we go back to the level of cohomology. Let $P\leq Q\leq S$. By Lemmas \ref{lemmahpqadecomposition} and \ref{lemmahpqbdecomposition} the composition $H^{n,m}(B)(\iota_P^Q)\circ H^{n,m}(A)(\iota_P^Q)$ is equal to
$$
H^n(tr)\circ H^m(tr')\circ H^m(\iota_P^Q)\circ H^n(\overline{\iota_P^Q}).
$$ 
Because cohomology over finite groups is a cohomological Mackey functor we know that $H^m(tr')\circ H^m(\iota_P^Q)$ is multiplication by $|Q\cap T|/|P\cap T|$. Moving out this factor we are left with $H^n(tr)\circ H^n(\overline{\iota_P^Q})$. As $\overline{\iota_P^Q}=\iota_{\overline P}^{\overline Q}$ we obtain again by properties of cohomology for finite groups that this composition is multiplication by $|\overline Q|/|\overline P|$. So finally we obtain that $H^{n,m}(B)(\iota_P^Q)\circ H^{p,q}(A)(\iota_P^Q)$ is multiplication by
$$
\frac{|Q\cap T|}{|P\cap T|}\frac{|\overline Q|}{|\overline P|}=\frac{|Q|}{|P|}.
$$
\end{proof}


\section{Construction of the spectral sequence.}\label{section_construction}

In this section we prove the main theorem of this paper:

\begin{thm}\label{maintheorem}
Let $\F$ be a fusion system over the $p$-group $S$, $T$ a strongly $\F$-closed subgroup of $S$ and $M$ a $\plocz$-module with trivial $S$-action. Then there is a first quadrant cohomological spectral sequence with second page
$$
E_2^{n,m}=H^n(S/T;H^m(T;M))^\F
$$
and converging to $H^{n+m}(\F;M)$.
\end{thm}

\begin{proof}
For each subgroup $P\leq S$ we have the short exact sequence 
$$
P\cap T\rightarrow P\rightarrow \overline P=P/P\cap T.
$$
The construction of the Lyndon-Hochschild-Serre spectral sequence in  \cite[XI.10.1]{homology} associates to this short exact sequence a double complex naturally isomorphic 
to the double complex 
$$
A^{n,m}(P)=\Hom_P(\Bar_{\overline P}^n\otimes \Bar_P^m,M)
$$
defined in Section \ref{section_amackey}. This double complex we can filter either by columns or rows. If we filter by columns we obtain a spectral sequence $\{{}^c E_k^{*,*}(P),d_k\}_{0\leq k\leq \infty}$ whose second page is ${}^c E_2^{n,m}(P)=H^n(\overline P;H^m(P\cap T;M))$. If we filter by rows we obtain a spectral sequence $\{{}^r E_k^{*,*}(P),d_k\}_{0\leq k\leq \infty}$ whose second page collapses as ${}^r E_2^{n,m}(P)=H^m(P;M)$ for $n=0$ and ${}^r E_2^{n,m}(P)=0$ for $n>0$.

For each morphism $\varphi\in \Hom_\F(P,Q)$ we have morphisms of double complexes $A^{n,m}(\varphi)\colon A^{n,m}(Q)\rightarrow A^{n,m}(P)$ and $B^{n,m}(\varphi)\colon A^{n,m}(P)\rightarrow A^{n,m}(Q)$ defined in Section \ref{section_amackey}. These morphisms of double complexes induce morphism of spectral sequences consisting of a sequence of morphism of differential bigraded $\plocz$-modules
$$
{}^c E_k^{*,*}(A)(\varphi)\colon {}^c E_k^{*,*}(Q)\rightarrow{}^c E_k^{*,*}(P)
$$
$$
{}^c E_k^{*,*}(B)(\varphi)\colon {}^c E_k^{*,*}(P)\rightarrow{}^c E_k^{*,*}(Q)
$$
and
$$
{}^r E_k^{*,*}(A)(\varphi)\colon {}^r E_k^{*,*}(Q)\rightarrow{}^r E_k^{*,*}(P)
$$
$$
{}^r E_k^{*,*}(B)(\varphi)\colon {}^r E_k^{*,*}(P)\rightarrow{}^r E_k^{*,*}(Q)
$$
for $0\leq k\leq \infty$. We deal now with the filtration by columns spectral sequences. The second page ${}^c E_2^{*,*}$ is obtained by computing vertical cohomology followed by horizontal cohomology in in the double complex $A^{*,*}$. Hence we have
$$
{}^c E_2^{n,m}(P)=H^n(\overline P;H^m(P\cap T;M)),
$$
$$
{}^c E_2^{n,m}(A)(\varphi)=H^{n,m}(A)(\varphi)\textit{ and }{}^c E_2^{n,m}(B)(\varphi)=H^{n,m}(B)(\varphi),
$$
for $P\leq S$ and $\varphi\in \Hom_\F(P,Q)$, where $H^{n,m}(A)$ and $H^{n,m}(B)$ are \emph{functors} $\F\rightarrow \plocz\dash mod$ by Remark \ref{remarkhpabisfunctor} and Corollary \ref{corhpqbisfunctor} respectively. Hence, for each $2\leq k\leq \infty$, we have a contravariant functor
$$
{}^c E_k^{*,*}(A):\F\rightarrow \textit{Differential bigraded $\plocz$-modules}
$$
and a covariant functor
$$
{}^c E_k^{*,*}(B):\F\rightarrow \textit{Differential bigraded $\plocz$-modules.}
$$
On the one hand, we can take invariants for each $2\leq k\leq \infty$ to obtain a differential bigraded $\plocz$-module
$$
{}^c {E_k^{*,*}}^\F=\{z\in {}^c E_k^{*,*}(S)|{}^c E_k^{*,*}(A)(\varphi)(z)={}^cE_k^{*,*}(A)(\iota_P^S)(z)\textit{ for $P\stackrel{\varphi}\rightarrow S$}\}.
$$
On the other hand, for $k=2$, we have by Proposition \ref{hpqisMackey} that $({}^c E_2^{*,*}(A),{}^c E_2^{*,*}(B))$ is a cohomological Mackey functor. Because ${}^c E_{k+1}^{*,*}=H^*({}^c E_k^{*,*},d_k)$ and because passing to cohomology preserves cohomological Mackey functors we deduce that ${}^c E_k^{*,*}(A)$ is a cohomological Mackey functor with covariant part ${}^c E_k^{*,*}(B)$ for $2\leq k<\infty$. By Proposition \ref{lemmacochainsforF} we obtain then that 
\begin{equation}\label{eqnbunchisss}
{{}^c E_{k+1}^{*,*}}^\F=H^*({}^c E_k^{*,*},d_k)^\F=H^*({{}^c E_k^{*,*}}^\F,d_k)
\end{equation}
for $2\leq k<\infty$. Fix now $n\geq 0$ and $m\geq 0$. For each  subgroup $P\leq S$ we have 
$$
{}^c E_k^{n,m}(P)={}^c E_{k+1}^{n,m}(P)=\ldots={}^c E_{\infty}^{n,m}(P)
$$
for $k$ big enough. Because there are a finite number of subgroups of $S$ we deduce that 
$$
{{}^c E_k^{n,m}}^\F={{}^c E_{k+1}^{n,m}}^\F=\ldots={}^c {E_{\infty}^{n,m}}^\F
$$
for $k$ big enough. Hence Equation (\ref{eqnbunchisss}) also holds for $k=\infty$ and we have obtained a spectral sequence 
$$
\{{{}^c E_k^{*,*}}^\F,d_k\}_{2\leq k\leq \infty}.
$$
To study whether this spectral sequence converges\textsl{•} recall that for $P\leq S$ the spectral sequence $\{{}^c E_k^{*,*}(P),d_k\}_{0\leq k\leq \infty}$ converges to $H^*(P;M)$. Hence we have short exact sequences 
$$
\xymatrix{
0\ar[r]& F^nH^{n+m}(P;M)\ar[r]&F^{n+1}H^{n+m}(P;M)\ar[r]&{}^c E^{n,m}_\infty(P)\ar[r]& 0
}
$$
where
$$
0\subseteq \ldots \subseteq F^nH^{n+m}(P;M)\subseteq F^{n+1}H^{n+m}(P;M)\subseteq \ldots\subseteq H^{n+m}(P;M)
$$
is the filtration induced on $H^*(P;M)$ by the filtration by columns on the double complex $A^{*,*}(P)$. This short exact sequence is natural with respect to morphisms of double complexes. Hence for each $\varphi\in \Hom_\F(P,Q)$ we have morphisms of short exact sequences
$$
\xymatrix{
0\ar[r]&F^nH^{n+m}(Q;M)\ar[d]^{F^nH^{n+m}(A)(\varphi)}\ar[r]&F^{n+1}H^{n+m}(Q;M)\ar[r]\ar[d]^{F^{n+1}H^{n+m}(A)(\varphi)}&{}^c E^{n,m}_\infty(Q)\ar[d]^{{}^c E_\infty^{n,m}(A)(\varphi)}\ar[r]& 0\\
0\ar[r]& F^nH^{n+m}(P;M)\ar[r]&F^{n+1}H^{n+m}(P;M)\ar[r]&{}^c E^{n,m}_\infty(P)\ar[r]& 0
}
$$
and
$$
\xymatrix{
0\ar[r]&F^nH^{n+m}(P;M)\ar[d]^{F^nH^{n+m}(B)(\varphi)}\ar[r]&F^{n+1}H^{n+m}(P;M)\ar[r]\ar[d]^{F^{n+1}H^{n+m}(B)(\varphi)}&{}^c E^{n,m}_\infty(P)\ar[d]^{{}^c E_\infty^{n,m}(B)(\varphi)}\ar[r]& 0\\
0\ar[r]& F^nH^{n+m}(Q;M)\ar[r]&F^{n+1}H^{n+m}(Q;M)\ar[r]&{}^c E^{n,m}_\infty(Q)\ar[r]& 0.
}
$$
We want to show that the morphisms $H^n(A)(\varphi)\colon H^n(Q;M)\rightarrow H^n(P;M)$ and $H^n(B)(\varphi)\colon H^n(P;M)\rightarrow H^n(Q;M)$ induced by $A$ and $B$ on the targets of the spectral sequences are the usual maps in cohomology of groups. We consider the total complex of the double complex $A^{*,*}(P)$ defined as usual by $\Tot^s(A)=\bigoplus_{n+m=s} A^{n,m}(P)$ and with total differential $d^h+d^v$. There is a chain map given by $\zeta\colon \Hom_P(\Bar_P^*,M)\rightarrow \Tot^*(A)$ sending $f\in \Hom_P(\Bar_P^m,M)$ to $\zeta(f)\in A^{0,m}$ defined by
$$
\zeta(f)(\overline{p}\otimes x)=f(x)\textit{, $\overline{p}\in \overline{P}$, $x\in \Bar_P^m$.}
$$
The map $\zeta$ induces an isomorphism between the cohomology of the total complex and $H^*(P;M)$ (cf. \cite[p. 352]{homology}). Now, from the definitions of the maps $A^{0,m}(\varphi)\colon A^{0,m}(Q)\rightarrow A^{0,m}(P)$ and $B^{0,m}(\iota_P^Q)\colon A^{0,m}(P)\rightarrow A^{0,m}(Q)$ it is easy to check that $H^n(A)(\varphi)$ and $H^n(B)(\iota_P^Q)$ are the usual maps in cohomology of groups (see \cite[(D), page 82]{brown}).

By properties of cohomology for finite groups $(H^n(A),H^n(B))\colon\F\rightarrow \plocz\dash mod$ is a  cohomological Mackey functor for each $n\geq 0$. Hence so are the functors $(F^nH^{n+m}(A),F^nH^{n+m}(B))\colon\F\rightarrow \plocz\dash mod$ induced in the filtration for $n,m\geq 0$. By the arguments above also the pair $({}^c E_\infty^{n,m}(A),{}^c E_\infty^{n,m}(B))\colon\F\rightarrow \plocz\dash mod$ is a cohomological Mackey functor for $n,m\geq 0$. Then by Lemma \ref{lemmainvariantsexactonMackey} we have a short exact sequence of $\plocz$-modules
$$
\xymatrix{
0\ar[r]& (F^nH^{n+m})^\F\ar[r]&(F^{n+1}H^{n+m})^\F\ar[r]&{{}^c E^{n,m}_\infty}^\F\ar[r]& 0.
}
$$
It is immediate that taking invariants and filtering commute and hence we have 
$$
\xymatrix{
0\ar[r]& F^n({H^{n+m}}^\F)\ar[r]&F^{n+1}({H^{n+m}}^\F)\ar[r]&{{}^c E^{n,m}_\infty}^\F\ar[r]& 0,
}
$$
for the filtration of ${H^{n+m}}^\F={H^{n+m}(S)}^\F$ given by
$F^n({H^{n+m}}^\F)=F^n({H^{n+m}(S)})\cap H^{n+m}(S)^\F$.
This finishes the proof.
\end{proof}
\begin{remark}
We have seen in the proof that for each $2\leq k\leq \infty$ the pair $$({}^c E_k^{*,*}(A),{}^c E_k^{*,*}(B))\colon \F\rightarrow \textit{Differential bigraded $\plocz$-modules}$$ is a cohomological Mackey functor. Moreover, $\{({}^c E_k^{*,*}(A),{}^c E_k^{*,*}(B))\}_{2\leq k\leq \infty}$ is a spectral sequence of Mackey functors that converges as a Mackey functor to the usual cohomology of finite groups Mackey functor $(H^*(A),H^*(B))\colon\F\rightarrow \plocz\dash mod$.
\end{remark}
%
%
%
%

\section{Comparison.}\label{section_examples}
In this section we compare our spectral sequence and Lyndon-Hochschild-Serre spectral sequence. Let $G$ be a finite group, $K\unlhd G$ and $S\in \Syl_p(G)$. Then $T=K\cap S$ is a Sylow $p$-subgroup of $K$. Moreover, $T$ is strongly $\F_S(G)$-closed.
Fix a $\plocz$-module $M$ with trivial $G$-action. The Lyndon-Hochschild-Serre spectral sequence $E_{*,G}$ of the extension $K\rightarrow G\rightarrow G/K$ is
$$
H^n(G/K;H^m(K;M))\Rightarrow H^{n+m}(G;M)
$$
meanwhile the spectral sequence $E_*$ from Theorem \ref{teoremaA} associated to $T$ is
$$
H^n(S/T;H^m(T;M))^{\F_S(G)}\Rightarrow H^{n+m}(\F;M).
$$
Note that by the classical stable elements theorem \cite[XII.10.1]{CE}, attributed to Tate by Cartan and Eilenberg, $H^*(G;M)=H^*(\F;M)$ and both spectral sequences converge to the same target.
Recall that, by construction, $E_*$ is a sub-spectral sequence of the Lyndon-Hochschild-Serre spectral sequence $E_{*,S}$ of $T\rightarrow S\rightarrow S/T$.

\begin{thm}\label{theorem:sscomparison}
The spectral sequences $E_{*,G}$ and $E_*$ are isomorphic.
\end{thm}
\begin{proof}
Consider the category $\F_G(G)$ with objects the subgroups of $G$ and morphisms given by $\Mor_{\F_G(G)}(H,I)=\Hom_G(H,I)$. Clearly $\F_S(G)$ is a full subcategory of $\F_G(G)$. For each subgroup $H\leq G$ we have a short exact sequence 
$$
H\cap K\rightarrow H\rightarrow \overline H=H/H\cap K.
$$
If $\varphi=c_g\colon H\rightarrow I$ is a morphism in $\F_G(G)$ then, as $K$ is normal in $G$, conjugation by $g\in G$ takes $H\cap K\rightarrow H\rightarrow \overline H$ to $I\cap K\rightarrow I\rightarrow \overline I$. Exactly the same construction of Section \ref{section_amackey} gives a cohomological Mackey functor $(A,B)\colon \F_G(G)\rightarrow \CCh^2(\plocz)$ with values $H\mapsto A^{n,m}(H)=\Hom_H(\Bar_{\overline H}^n\otimes \Bar_H^m,M)$, where $\Bar_H^*$ and $\Bar_{\overline H}^*$ are the bar resolutions for $H$ and $\overline{H}$ respectively. Moreover, for $H\leq S$, as $T=K\cap S$, we have $H\cap K=H\cap T$ and this functor over $\F_G(G)$ extends the one built in Section \ref{section_amackey} over $\F_S(G)$.

The inclusion of the short exact sequence $T\rightarrow S\rightarrow S/T$ into $K\rightarrow G\rightarrow G/K$ gives a morphism $\{res_r\}_{r\geq 2}$ of spectral sequences from $E_{*,G}$ into $E_{*,S}$. The morphism of differential graded algebras $res_2\colon E_{2,G}\rightarrow E_{2,S}$ coincides with the morphism  induced in cohomology by the functor $A$ applied to the inclusion morphism $S\leq G$ of $\F_G(G)$, $H^{*,*}(A)(\iota_S^G)$. Applying the functor $B$ to the same inclusion $S\leq G$ we get another morphism  going in the opposite direction (transfer)
$$
\xymatrix{
H^n(G/K;H^m(K;M))\ar@(ur,ul)[rrr]^{H^{n,m}(A)(\iota_S^G)} &&&H^n(S/T;H^m(T;M))\ar@(dl,dr)[lll]^{H^{n,m}(B)(\iota_S^G)}.
}
$$
Recall that $E_2\leq E_{2,S}$ are exactly the $\F$-stable elements $H^p(S/T;H^q(T;M))^{\F_S(G)}$. Because conjugation by $g\in G$ induces the identity on $H^p(G/K;H^q(K;M))$ it is straightforward that $res_2(E_{2,G})\leq E_2$. Hence $\{res_r\}_{r\geq 2}$ is a morphism of spectral sequences $E_{*,G}\rightarrow E_*$. If we prove that $res_2(E_{2,G})=E_2$ then $res_2$ is an isomorphism and hence $res_r$ is an isomorphism for each $r\geq 2$ and we are done.
To see that $res_2(E_{2,G})=E_2$ we proceed as usual when there is a Mackey functor available (cf. \cite[Theorem III.10.3]{brown}). Let $z\in H^n(S/T;H^m(T;M))^{\F_S(G)}$ and consider $w=H^{n,m}(B)(\iota_S^G)(z)\in H^n(G/K;H^m(K;M))$. By the double coset formula \ref{defMackey}(3) and the cohomological condition \ref{defMackey}(4) and because $z$ is $\F_S(G)$-stable we obtain
\begin{align}
H^{n,m}(A)(\iota_S^G)(w)=&\sum_{x\in S\backslash G/S} B(\iota^S_{S\cap {}^xS})A(\iota^{{}^xS}_{S\cap {}^xS})c_{x^{-1}}^*(z)\nonumber\\
=&\sum_{x\in S\backslash G/S} B(\iota^S_{S\cap {}^xS}) A(\iota^{S}_{S\cap {}^xS})(z)\nonumber\\
=&\sum_{x\in S\backslash G/S}|S:{S\cap {}^xS}|z=|G:S|z.\nonumber
\end{align}
As $q=|G:S|$ is a $p'$-number it follows that $z=res(\frac{w}{q})$.
\end{proof}

\begin{exa}
Consider the symmetric group on $6$ letters $S_6$. It has Sylow $2$-subgroup $S=C_2\times D_8$ where $D_8$ is the dihedral group of order $8$. Because $A_6\unlhd S_6$, the subgroup $T=S\cap A_6=D_8$ is strongly closed in $\F=\F_S(S_6)$. In this example we describe the Lyndon-Hochshild-Serre spectral sequence of $A_6\rightarrow S_6\rightarrow C_2$ interpreted as the spectral sequence $E_*^{*,*}$ of Theorem \ref{teoremaA} applied to $\F$ and $T$. This demonstrates how the new spectral sequence works.

In the fusion system $\F$ there are three $\F$-centric an $\F$-radical subgroups, namely, $S$, $P=C_2^3$ and $Q=C_2^3$. The intersections $P\cap T$ and $Q\cap T$ are the two Klein subgroups of $T=D_8$. The automorphisms are $\Aut_\F(S)=1$ and $\Aut_\F(P)\cong\Aut_\F(Q)\cong S_3$, the symmetric group on $3$ letters. 

Denote by $E_{*,S}^{*,*}$, $E_{*,P}^{*,*}$ and $E_{*,Q}^{*,*}$ the Lyndon-Hochschild-Serre spectral sequences of the extensions $T\rightarrow S\rightarrow C_2$, $P\cap T\rightarrow P\rightarrow C_2$ and $Q\cap T\rightarrow Q\rightarrow C_2$ respectively. All  three extensions are direct products and hence all differentials are $0$ and the three spectral sequences collapse at the second page. In particular, the ring $H^*(S;\FF_2)$ is isomorphic as a ring to $E^{*,*}_{2,S}$ and hence $H^*(S_6;\FF_2)$ is isomorphic as a ring to $E_2^{*,*}$. Moreover, for the invariants we have
\begin{equation}\label{equ:invariantsexample}
E_2^{*,*}={E_{2,S}^{*,*}}^\F=E_{2,S}^{*,*}\cap (res^S_P)^{-1}({E_{2,P}^{*,*}}^{S_3})\cap
(res^S_Q)^{-1}({E_{2,Q}^{*,*}}^{S_3})\nonumber,
\end{equation}
because it is enough to consider invariants with respect to $\F$-centric and $\F$-radical subgroups by Alperin's fusion theorem. Here, $res^S_P\colon E_{2,S}^{*,*}\rightarrow E_{2,P}^{*,*}$ and $res^S_Q\colon E_{2,S}^{*,*}\rightarrow E_{2,Q}^{*,*}$ are the restriction maps. Denoting by subscripts the degrees we have the following:
\begin{align}
E^{*,*}_{2,S}=&H^*(D_8;\FF_2)\otimes H^*(C_2;\FF_2)=\FF_2[x_1,y_1,w_2]/(xy)\otimes \FF_2[u],\nonumber\\
E^{*,*}_{2,P}=&H^*(C_2^2;\FF_2)\otimes H^*(C_2;\FF_2)=\FF_2[x_1,x'_1]\otimes \FF_2[u],\textit{ and}\nonumber\\
E^{*,*}_{2,Q}=&H^*(C_2^2;\FF_2)\otimes H^*(C_2;\FF_2)=\FF_2[y_1,y'_1]\otimes \FF_2[u]\nonumber.
\end{align}
Restrictions are given by 
\begin{align}
res^S_P(x)=x\textit{, }&res^S_P(y)=0\textit{, } res^S_P(w)=xx'+x'^2 \textit{, }res^S_P(u)=u\textit{ and }\nonumber\\
res^S_Q(x)=0\textit{, }&res^S_Q(y)=y\textit{, } res^S_Q(w)=yy'+y'^2\textit{, }res^S_Q(u)=u\nonumber.
\end{align}
Now $S_3=\Aut_{S_6}(P)$ acts on $P\cap T=C_2^2$ and on the quotient $C_2=P/P\cap T$. The induced action on $H^*(C_2^2;\FF_2)$ is the natural one and on $H^*(C_2;\FF_2)$ the only possibility is the trivial action. Hence, the invariants are given by 
$$
{E_{2,P}^{*,*}}^{S_3}=\FF_2[x_1,x'_1]^{S_3}\otimes \FF_2[u]=\FF_2[d_2,d_3]\otimes \FF_2[u],
$$
where $d_2=x^2+x'^2+xx'$ and $d_3=(x+x')xx'$ are the Dickson's invariants. Analogously, we have that
$$
{E_{2,Q}^{*,*}}^{S_3}=\FF_2[e_2,e_3]\otimes \FF_2[u]
$$
with $e_2=y^2+y'^2+yy'$ and $e_3=(y+y')yy'$. It is straightforward that
$$
d_2=res^S_P(x^2+w)\textit{, }d_3=res^S_P(xw)\textit{, } e_2=res^S_Q(y^2+w)\textit{ and } e_3=res^S_Q(yw).
$$
From this, it is immediate that $\FF_2[x^2+y^2+w,xw,yw]\otimes \FF_2[u]\subseteq E_2^{*,*}$.

To check the reversed inclusion we first consider stable elements in the polynomial algebras $\FF_2[x,w]$ and $\FF_2[y,w]$. As in \cite[Lemma 1.4.6]{Long2008}, the restrictions $res^S_P|_{\FF_2[x,w]}$ and $res^S_Q|_{\FF_2[y,w]}$ are injective and so $\FF_2[x,w]\cap {(res^S_P)^{-1}({E_{2,P}^{*,*}}^{S_3})}=\FF_2[x^2+w,xw]$ and $\FF_2[y,w]\cap {(res^S_Q)^{-1}({E_{2,Q}^{*,*}}^{S_3})}=\FF_2[y^2+w,yw]$. A class $v$ of $H^n(D_8;\FF_2)$ can be written as follows, where we set $k=[\frac{n}{2}]$:
$$
v=\sum_{i=0}^{k}\alpha_iw^ix^{n-2i}+\beta_iw^iy^{n-2i}.
$$
From the discussion above we have that if $v$ is $\F$-invariant then
$$
v=\epsilon w^k+\sum_{2i+3j=n}\gamma_i(x^2+w)^i(xw)^j+\delta_i(y^2+w)^i(yw)^j,
$$
where $\epsilon=0$ for $n$ odd and $\epsilon=\alpha_k+\beta_k=\gamma_k=\delta_k$ if $n$ is even. If $n$ is odd, then the equalities $(x^2+y^2+w)(xw)=(x^2+w)(xw)$ and $(x^2+y^2+w)(yw)=(y^2+w)(yw)$ give that $v\in \FF_2[x^2+y^2+w,xw,yw]$. If $n$ is even, then the only term left to consider is 
$$
\gamma_k(x^2+w)^k+\delta_k(y^2+w)^k+\epsilon w^k=\gamma_k((x^2+w)^k+(y^2+w)^k+w^k)
$$
and an easy induction shows that $(x^2+w)^k+(y^2+w)^k+w^k=(x^2+y^2+w)^k$.
So  $E_2^{*,*}=\FF_2[x^2+y^2+w,xw,yw]\otimes \FF_2[u]$. The corner of $E_2^{*,*}$ is 
$$
\xymatrix@=1pt{
&\\
&xw, yw& xwu,ywu&xwu^2,ywu^2\\
&x^2+y^2+w&(x^2+y^2+w)u&(x^2+y^2+w)u^2\\
&0&0&0\\
&1& u& u^2&\\
\ar@{.}"1,2";"6,2"\ar@{.}"5,1";"5,5"&
}
$$
and we deduce that $H^*(S_6;\FF_2)=\FF_2[u_1,a_2,b_3,c_3]/(bc)$ with generators $a=x^2+y^2+w$, $b=xw$ and $c=yw$.
\end{exa}

\section{Stallings' Theorem.}\label{section_Stallings}

Associated to every first quadrant spectral sequence there is a five terms exact sequence. In the case of the Lyndon-Hochschild-Serre spectral sequence for $K\unlhd G$ and the $G$-module $M$ we obtain the inflation-restriction exact sequence:
\begin{equation}\footnotesize\label{5termsesgroup}
0\rightarrow H^1(G/K;M^K)\rightarrow H^1(G;M)\rightarrow H^1(K,M)^{G/K}\rightarrow H^2(G/K;M^K)\rightarrow H^2(G;M),
\end{equation}
where the second arrow from the right is the transgression. Before introducing the five terms exact sequence for the spectral sequence of Theorem \ref{teoremaA} we introduce some notation. So let $\F$ be a fusion system over the $p$-group $S$ with a strongly closed $\F$-subgroup $T$. Set $[T,\F]=\langle [t,\varphi]| t\in T\textit{, }\varphi\in \Hom_\F(\langle u\rangle,T)\rangle\leq T$, where $[t,\varphi]=t\varphi(t^{-1})$, $T^p=\langle t^p, t\in T\rangle$, which is characteristic in $T$, and the commutator subgroup $[T,S]=\langle t^{-1}s^{-1}ts|t\in T\textit{ and } s\in S\rangle\unlhd T$. Because the element-wise product $T^p[T,S]$ is a normal subgroup of $T$, the element-wise product $T^p[S,T]R$ is a subgroup of $T$ for any $R\leq T$. For instance, $T^p[T,S][T,\F]=T^p[T,\F]\leq T$.

The five terms exact sequence for the spectral sequence of Theorem \ref{teoremaA} for $\F$, $T$ and the $\plocz$-module $M$ with trivial $S$-action is the following:
\begin{equation}\small\label{5termses}
0\rightarrow H^1(S/T;M)^\F\rightarrow H^1(\F;M)\rightarrow H^1(T;M)^{\F}\rightarrow H^2(S/T;M)^\F\rightarrow H^2(\F;M),
\end{equation}
where the arrow $H^1(T;M)^{\F}\rightarrow H^2(S/T;M)^\F$ is the transgression. 

For coefficients $M=\FF_p$ we have 
\begin{equation}\label{equdiagram2}
H^1(\F;\FF_p)=H^1(S;\FF_p)^\F=\Hom(S/S^p[S,\F],\FF_p)
\end{equation}
and
\begin{equation}\label{equdiagram3}
H^1(T;\FF_p)^\F=\Hom(T/T^p[T,\F],\FF_p).
\end{equation}
We also have
\begin{equation}\label{equdiagram4}
H^1(S/T;M)^\F=H^1(S/T;M)^{\F/T}\textit{ and } H^2(S/T;M)^\F=H^2(S/T;M)^{\F/T}
\end{equation}
by Remark \ref{remark_Cravenquotient}. 

If $\F_i$ is a fusion system over the $p$-group $S_i$ for $i=1,2$, a homomorphism of groups $\phi:S_1\rightarrow S_2$ is fusion preserving if for each $\varphi\in \Hom_{\F_1}(P,S_1)$ there exists $\hat\varphi\in \Hom_{\F_2}(\phi(P),S_2)$ such that $\phi\circ \varphi=\hat\varphi\circ \phi$. It is easy to see that such a homomorphism induces a map in cohomology $H^*(\F_2;\FF_p)\rightarrow H^*(\F_1;\FF_p)$. In fact, by the work of Ragnarsson \cite{Kari}, it induces a map even at the level of stable classifying spaces. Assume, in addition, that $\phi$ induces a map of short exact sequences
$$
\xymatrix{
T_1\ar[r]\ar[d]&S_1\ar[r]\ar[d]^{\phi}&S_1/T_1\ar[d]\\
T_2\ar[r]&S_2\ar[r]&S_2/T_2,\\
}
$$
where $T_i$ is strongly closed in $S_i$ with respecto to $\F_i$ for $i=1,2$. This is equivalent to assume that $\phi(T_1)\leq T_2$. Denote by $E_i$ the spectral sequence from Theorem \ref{teoremaA} applied to the strongly closed subgroup $T_i$ for $i=1,2$. Then $\phi$ induces a morphism of spectral sequences $E_2\rightarrow E_1$ and, in particular, a map of five terms exact sequences.
\begin{thm}[\cite{stallings}]\label{thm:stallings}
Let $\F_i$ be a fusion system over the $p$-group $S_i$ for $i=1,2$ and let $\phi\colon S_1\rightarrow S_2$ be a fusion preserving homomorphism. Define $S_{i,0}=S_i$ and $S_{i,n+1}=S_{i,n}^p[S_{i,n},\F_i]$ for $i=1,2$ and $n\geq 0$. If the induced map in cohomology $H^i(\F_2;\FF_p)\rightarrow H^i(\F_1;\FF_p)$ is isomorphism for $i=1$ and monomorphism for $i=2$ then $S_1/S_{1,n}\cong S_2/S_{2,n}$ for each $n\geq 1$. In particular, for $n$ big enough we obtain that $S_1/\O^p_{\F_1}(S_1)\cong S_2/\O^p_{\F_2}(S_2)$.
\end{thm}
\begin{proof}
We will prove by induction that $S_1/S_{1,n}\cong S_2/S_{2,n}$ and that $S_{i,n}$ is strongly $\F_i$-closed and contains $O^p_{\F_i}(S_i)$ for $i=1,2$. For the base case $n=1$, we have that $S_{i,1}$ contains $O^p_{\F_i}(S_i)$ and is strongly $\F_i$-closed by \cite[Corollary A.6]{tateyoshida} ($i=1,2$). Moreover, by hypothesis, $H^1(\F_2;\FF_p)\cong H^1(\F_1;\FF_p)$ and then by Equation (\ref{equdiagram2}) we get $S_1/S_{1,1}\cong S_2/S_{2,1}$.

Now let $n\geq 1$. As $\Phi$ is fusion preserving it is easy to see that $\phi(S_{1,n})\leq S_{2,n}$. Then we have the following map of short exact sequences:
$$
\xymatrix{
S_{1,n}\ar[r]\ar[d]&S_1\ar[r]\ar[d]^{\phi}&S_1/S_{1,n}\ar[d]\\
S_{2,n}\ar[r]&S_2\ar[r]&S_2/S_{2,n}.\\
}
$$

By the induction hypothesis, $S_{1,n}$ and $S_{2,n}$ are strongly closed in $\F_1$ and $\F_2$ respectively. Then by the discussion before the theorem we have a map of five terms short exact sequences in cohomology with trivial $\FF_p$-coefficients:
$$
\small
\xymatrix@=10pt{
0\ar[r]&H^1(S_1/S_{1,n})^{\F_1}\ar[r]\ar[d]^{f_1}&H^1(\F_1)\ar[r]\ar[d]^{g_1}&H^1(S_{1,n})^{\F_1}\ar[r]\ar[d]^{h_1}& H^2(S_1/S_{1,n})^{\F_1}\ar[r]\ar[d]^{f_2}& H^2(\F_1)\ar[d]^{g_2}\\
0\ar[r]&H^1(S_2/S_{2,n})^{\F_2}\ar[r]&H^1(\F_2)\ar[r]&H^1(S_{2,n})^{\F_2}\ar[r]& H^2(S_2/S_{2,n})^{\F_2}\ar[r]& H^2(\F_2).
}
$$
Because $O^p_{\F_i}(S_i)$ is contained in $S_{i,n}$ the quotient $\F/S_{i,n}$ is a $p$-group, i.e.,  $\F/S_{i,n}=\F_{S/S_{i,n}}(S/S_{i,n})$ for $i=1,2$. Then, by Equation (\ref{equdiagram4}), the maps $f_1$ and $f_2$ are isomorphisms as $S_1/S_{1,n}\cong S_2/S_{2,n}$. Now, by hypothesis, $g_1$ is an isomorphism and $g_2$ is a monomorphism. Hence by the five lemma $h_1$ is an isomorphism. Then by Equation (\ref{equdiagram3}) we obtain that  $S_{1,n}/S_{1,n+1}\cong S_{2,n}/S_{2,n+1}$ and hence $S_1/S_{1,n+1}\cong S_2/S_{2,n+1}$.

To finish the induction step, denote by $\F_{i,n}$ the unique p-power index fusion subsystem of $\F_i$ on $S_{i,n}$ \cite[Theorem 4.3]{extensions}. Then using \cite[Lemma A.5]{tateyoshida} we obtain that $S_{i,n+1}=S_{i,n}^p[S_{i,n},S_i]O^p_{\F_{i,n}}(S_{i,n})$ and hence, by \cite[Corollary A.14]{tateyoshida}, $S_{i,n+1}=S_{i,n}^p[S_{i,n},S_i]O^p_{\F_i}(S_i)$. Then $S_{i,n+1}$ contains $O^p_{\F_i}(S_i)$ and by \cite[Proposition A.7(1)]{tateyoshida} $S_{i,n+1}$ is strongly $\F_i$-closed for $i=1,2$.

For the second part of the statement recall that for any finite $p$-group $R$ the series $R_0=R$, $R_n=R_{n-1}^p[R_{n-1},R]$ ($n\geq 1$) becomes trivial for $n$ big enough. Then, considering the image of $S_{i,n}$ in $S_i/O^p_\F(S_i)$, it is easy to see that $S_{i,n}=O^p_\F(S_i)$ for $n$ big enough and $i=1,2$.
\end{proof}
\begin{cor}[{\cite[7.2.5]{evens}}]
Let $\F$ be a fusion system over the $p$-group $S$. If the map $H^2(\F/E^p_\F(S);\FF_p)\rightarrow H^2(\F;\FF_p)$ is a monomorphism then $S/O^p_\F(S)$  is elementary abelian.
\end{cor}
\begin{proof}
Set $\F_1=\F$ and $\F_2=\F/E^p_\F(S)$ and consider the fusion preserving quotient map $\F_1\rightarrow \F_2$. By Equation (\ref{equdiagram2}) and because $E^p_\F(S)=\Phi(S)O^p_\F(S)=S^p[S,\F]$, the quotient map induces an isomorphism in degree $1$ cohomology. Then Theorem \ref{thm:stallings} gives that $O^p_\F(S)=E^p_\F(S)$ and we are done.
\end{proof}

\begin{cor}[\cite{Tate}]
Let $\F$ be a fusion system over the $p$-group $S$. If the restriction map $H^1(\F;\FF_p)\rightarrow H^1(S;\FF_p)$ is an isomorphism then $\F=\F_S(S)$.
\end{cor}
\begin{proof}
Consider $\F_1=\F_S(S)$, $\F_2=\F$ and the fusion preserving morphism given by inclusion $\F_1\subseteq \F_2$. Then $H^1(\F;\FF_p)\rightarrow H^1(S;\FF_p)$ is isomorphism by hypothesis and $H^2(\F;\FF_p)\rightarrow H^2(S;\FF_p)$ is monomorphism by definition. Then from Theorem \ref{thm:stallings} we obtain $O^p_\F(S)=1$. Thus there are no $p'$-automorphisms in $\F$ and $\F=\F_S(S)$.
\end{proof}

\end{document}